\documentclass[12pt]{amsart}

\title{Some terminal orders on 3-folds}

\author{Daniel Chan}
\thanks{The authors were supported by the Australian Research Council Discovery Project grant  DP220102861}
\address{School of Mathematics and Statistics, 
UNSW Sydney, 
NSW 2052,
Australia
}
\email{danielc@unsw.edu.au}

\author{Colin Ingalls}

\usepackage{mymacros,amssymb}
\usepackage{amsmath}
\usepackage{amsfonts}
\usepackage{amscd}
\usepackage{hyperref}

\thanks{}

\begin{document}

\begin{abstract}
We produce the first known examples of non-trivial terminal local orders in the 3-dimensional case. We exhibit two different constructions. The first produces maximal orders with given toric ramification data. Since toric terminal ramification data in dimension 3 has been classified this produces toric examples. The second construction yields maximal orders ramified on surfaces with Kleinian singularities. These are automatically terminal. 
\end{abstract}

\maketitle

Throughout, $k$ denotes a field of characteristic zero that contains $\ell^{th}$ roots of unity $\mu_\ell$.

\section{Introduction} \label{sec:intro}

One of the highlights of algebraic geometry is Mori's classification of terminal singularities in dimension three. 
There is now a version of Mori's minimal model program for orders on varieties \cite{CI05}, \cite{11authors}. Much is known in dimension two, there is a complete classification of terminal, canonical and klt orders and their ramification data \cite{CHI}, \cite{CI21} (see Section~\ref{sec:back} for a review of these notions). By contrast, we know next to nothing in dimension three. The only previously known terminal orders were the commutative terminal singularities.  
Even examples of terminal ramification data was limited to the toric case \cite{11authors,basilThesis} where the combinatorics of toric geometry and toric studies of the Brauer group can be used. In this paper, we search for terminal orders by exploring two cases, i) toric ramification data and ii) when the ramification is on a Kleinian singularity or similarly mild singularity. 

We start by investigating toric ramification data over the polynomial ring $k[x_1,\ldots, x_n]$. In terms of orders, this means we are considering maximal orders over $k[x_1,\ldots,x_n]$ in a central simple $K$-algebra $Q$ where $K = k(x_1,\ldots, x_n)$ and $Q$ is a tensor product of symbols of the form $(x_i,x_j)^K_{\zeta}$ for $\zeta$ some $\ell$-th root of unity. We have the following construction by Proposition~\ref{prop:AQisMaximal}. 
\begin{theorem}
There exists a cross product algebra 
$$A_Q = k[y_1,\ldots, y_n] *_{c_Q} (\bZ/l\bZ)^{n(n-2)}$$ 
which is a maximal order in a central simple $K$-algebra, Morita equivalent to $Q$.
\end{theorem}
In fact, we explicitly compute the 2-cocycle $c_Q$ in the proposition. This produces toric terminal orders corresponding to all the toric 3-dimensional ramification data of odd prime index $l$, classified in \cite[Proposition~5.4]{11authors}.

Turning to our second class of examples, we first note that an easy sufficient criterion for ramification data to be terminal is that the associated log variety $(X,D)$ is log terminal. In particular, this is the case if $X$ is smooth complete local of dimension 3, and $D$ has a canonical singularity so has the form $\Spec k[[u,v]]^G$ for some finite subgroup $G < SL_2$. The ramification of an order with such an associated log variety is given by a cyclic extension $\tilde{D} \to D$ subject to the following two obstructions arising from the coniveau spectral sequence \cite[Section~1.2]{CTHK} for the Brauer group of a variety. Firstly, $\Dtilde \to D$ should be \'etale in codimension one. This naturally suggests that one considers normal subgroups $N < G$ with cyclic quotient $G/N$ and set $\Dtilde = \Spec k[[x,y]]^N$. We call such ramification data {\em Kleinian}. It turns out the second obstruction vanishes so the interesting question is what are the corresponding maximal orders with these terminal ramification data.

We construct examples of such orders in all cases except one.  
\begin{theorem}  \label{thm:mainKleinian}
With the exception of the degree 4 cover $\tilde{D} \to D$ of a type $D_{2n+3}$ Kleinian singularity, every Kleinian ramification data of index $\ell$ is exhibited by a maximal order of the form 
\begin{equation*}
    \frac{k[[u,v,w]] \langle x, y \rangle}{(x^{\ell} -a, y^{\ell}-b, yx - \zeta xy - r)}
\end{equation*}
where $a,b,r \in k[[u,v,w]]$ and $\zeta \in \mu_{\ell}$. 
\end{theorem}

The next natural question would of course be to classify maximal orders with such Kleinian ramification. This seems quite a hard question so we pose the simpler question of uniqueness under the assumption of the order being Cohen-Macaulay. We obtain the following from Corollary~\ref{cor:unique} and Remark~\ref{rem:typeAunique}.

\begin{theorem}
Let $Q$ be division ring, finite dimensional over its centre $k((u,v,w))$ and ramified on a type $A_{\ell-1}$ Kleinian singularity with ramification given by the smooth cover. Then any Cohen-Macaulay maximal order in $Q$ is isomorphic to 
\begin{equation*}
        \frac{k[[u,v,w]] \langle x, y \rangle}{(x^{\ell} -u, y^{\ell}-v, yx - \zeta xy - w)}
\end{equation*}
where $\zeta$ is a primitive $\ell$-th root of unity.  
\end{theorem}

\section{Background on Ramification of Orders and MMP}  \label{sec:back}

In this section, we briefly review and give references for background material on ramification for orders on varieties and the minimal model program for them. More details can be found in \cite{CI05}, \cite{11authors} or \cite[Sections~2.2 and 2.3]{CHI}

Let $X$ be a normal variety with function field $K$ and $Q$ a central simple $K$-algebra. An {\em order on $X$ in $Q$} is a coherent subsheaf of $\cO_X$-algebras $A$ of $Q$ such that $KA = Q$. These can be ordered by inclusion so one may speak of {\em maximal orders} with respect to this partial order. 

Given a maximal order $A$, the {\em ramification} of $A$ along some prime divisor $D$ measures the failure of $A$ being Azumaya there. There are several ways to define it, the most useful in this paper is the following. For more details and proofs, the reader should consult \cite{AG}, \cite{CI05}. Let $\hat{A}_D$ be the complete localisation of $A$ at $D$. Then $\rad \hat{A}_D$ is a principal left and right ideal and in fact, we can find $\pi \in \rad \hat{A}_D$ such that $\pi \hat{A}_D = \hat{A}_D \pi = \rad \hat{A}_D$. If $\kappa(D)$ denotes the residue field of $D$, then $\bar{A}_D:= \hat{A}_D/ \pi \hat{A}_D$ is a simple $\kappa(D)$-algebra so $\hat{A}_D$ is Azumaya if and only if the centre $\tilde{\kappa}_D:= Z(\bar{A}_D) = \kappa(D)$. In general though, $\tilde{\kappa}_D$ is a cyclic field extension of $\kappa(D)$ where the action of a generator of the Galois group is given by $\alpha \mapsto \pi \alpha \pi^{-1}$. The {\em ramification of $A$ along $D$} is the cyclic field extension $\tilde{\kappa}_D/\kappa(D)$. We will also call the normalisation of $D$ in $\tilde{\kappa}_D$, the {\em ramification cover}. The degree $e_D$ of this cyclic extension is the called the {\em ramification index}. 

A maximal order $A$ can only ramify on a finite set of divisors, so it makes sense to attach the log variety $(X,\Delta_A)$ where 
\begin{equation*}
    \Delta_A := \sum_D \left(1 - \frac{1}{e_D}D\right) \quad \text{and so } \ \ K_A := K_X + \Delta_A.
\end{equation*}

One can now define the notion of discrepancy and thus terminal, canonical, log terminal etc. In particular, we have the following.

\begin{definition}  \label{def:terminal}
A maximal order $A$ on $X$ is {\em terminal} if $K_A$ is $\bQ$-Cartier and given any proper birational morphism $\pi \colon \Xtilde \to X$ and maximal order $\Atilde$ containing $\pi^*A$ we have 
\begin{equation*}
    K_{\Atilde} - \pi^* K_A \geq 0. 
\end{equation*}
\end{definition}

\section{Some terminal orders with toric ramification data}

In this section, we construct some maximal orders over the polynomial ring $k[x_1,\ldots, x_n]$ with ``toric'' ramification data. We also compute their ramification data. 

Let $R$ be a commutative noetherian normal domain. We will use {\em symbols} which we recall are $R$-algebras of the following form. Let $a,b \in R-0$ and $\zeta \in R$ be a primitive $\ell$-th root of unity. Then we let 
$$(a,b)^R_{\zeta} : = R\langle y,z \rangle /(y^{\ell} - a, z^{\ell} - b, zy - \zeta yz)$$
and refer to $y,z$ as {\em standard generators}
for $(a,b)^R_{\zeta}$. Symbols are Azumaya along the locus where $a,b$ are units. In particular, if $K$ is the fraction field of $R$ then note that $(a,b)^K_{\zeta} \simeq (a,b)^R_{\zeta}\otimes_R K$ is a central simple $K$-algebra. 

Let us now fix a base field $k$ that contains primitive $\ell$-th roots of unity. We denote the group of $\ell$-th roots of unity by $\mu_\ell \subset k$. Let $Q = (q_{ij})$ be an $n \times n$-matrix with entries in $\mu_\ell$ which is skew-symmetric in the sense that $q_{ji} = q_{ij}^{-1}$ and $q_{ii}=1$. Shrinking $\ell$ if necessary, we may assume the $q_{ij}$ generate $\mu_{\ell}$.

Write $R = k[x_1,\ldots,x_n]$ with field of fractions $K$. Consider first the tensor product of symbol algebras determined by $Q$,
\begin{equation*}  \label{eq:tensorQsymbols}
\Lambda_Q:=\bigotimes_{1 \leq i < j \leq n} (x_i,x_j)^R_{q_{ij}}     
\end{equation*}
and let $y_{ij},y_{ji}$ denote the standard generators of $(x_i,x_j)^R_{q_{ij}}$ so 
\begin{equation*}  \label{eq:yrs}
y_{rs}^{\ell} = x_r \ \text{for all } \ r,s. 
\end{equation*}
Amongst the various $\ell$-th roots of $x_i$, we will choose $y_i := y_{i,i+1}$ where the indices are taken modulo $n$. For $j \neq i, i+1$, we note $y_{ij}$ and $y_i$ commute so we have $\ell$-th roots of unity 
\begin{equation*}  \label{eq:tauij}
    \tau_{ij} := y_{ij}y_i^{-1}  \in K \otimes_R \Lambda_Q.
\end{equation*}
Note that as all the $y_i$'s and $\tau_{ij}$'s skew-commute by roots of unity in $\mu_{\ell}$ we see that $\langle \tau_{ij} \rangle / \mu_{\ell}\simeq (\bZ/\ell \bZ)^{n(n-2)}$, that is, the $n(n-2)$ many $\tau_{ij}$'s generate a $\mu_{\ell}$-central extension of $(\bZ/\ell\bZ)^{n(n-2)}$.
This can be recorded in a 2-cocycle $c_Q \in H^2((\bZ/\ell\bZ)^{n(n-2)}, \mu_{\ell})$. 

Now $\Lambda_Q$ is an order, but not a maximal one since the order 
\begin{equation*}  \label{eq:maxToricOrder}
A_Q := k[y_1,\ldots, y_n] \langle \tau_{ij} | j \neq i, i+1\rangle
\end{equation*}
clearly contains it. Note also that $\Lambda_Q$ is Azumaya away from the coordinate hyperplanes in $\Spec R$ so the ramification is in this sense ``toric''. Though the order $A_Q$ {\it a priori} depends on the choice of the $y_i$, replacing $y_i$ with any other $y_{ij}$ would not change the order. Also, the group generated by the $\tau_{ij}$ contains all the $y_{ir}y_{is}^{-1}$ for all $i,r,s$. 

\begin{proposition}  \label{prop:AQisMaximal}
Suppose that all the $q_{ij}$ are primitive $\ell$-th roots of unity. Then the order $A_Q$ above is maximal and isomorphic to the crossed product algebra 
$$ k[y_1,\ldots, y_n] *_{c_Q} (\bZ/\ell \bZ)^{n(n-2)}$$
where $c_Q$ is the 2-cocycle defined above. It has global dimension $n$. 

Lastly, writing $J := \{1,\ldots, n\} - \{i\}$, the ramification above the coordinate hyperplane $x_i = 0$ is given by the cyclic cover of $k[x_j| j \in J]$ defined by
\begin{equation} \label{eq:ramAQ}
    k[y_{ji}| j \in J]^H  \ \text{where } \ H=\{(m_j)_{j \in J} \in (\bZ/\ell \bZ)^J | \sum m_j = 0 \} 
\end{equation}
and $(\bZ/\ell \bZ)^J$ acts on $k[y_{ji}| j \in J]$ by $(m_r)_{r \in J}. y_{ji} = q_{ji}^{m_j} y_{ji}$. 
\end{proposition}
\begin{proof}
Using skew-commutativity of the generators, one readily checks that $A_Q$ is isomorphic to the crossed product algebra above so, by Maschke's theorem, has global dimension $n$. 

Let $R_i$ denote the local ring of $R$ at the generic point of the hyperplane $x_i = 0$. Maschke's theorem also shows that $A_Q \otimes_R R_i$ is hereditary, so maximality of $A_Q$ will follow if we can show that the top of $A_Q \otimes_R R_i/(x_i)$ is a simple algebra. Now $y_i$ is a normal element of $A_Q$ which is nilpotent modulo $x_i$ and expressing $A_Q/(y_i)$ as a crossed product algebra shows that the top of $A_Q \otimes_R R_i/(x_i)$ is $A_Q/(y_i) \otimes_R k(x_j| j \in J)$. It thus suffices now to show that the centre of $A_Q/(y_i)$ is the invariant ring $k[y_{ji}| j \in J]^H$ in (\ref{eq:ramAQ}) above.

To simplify calculations, we will work in $\bar{A}_i := \left(A_Q/(y_i)\right) [x_j^{-1} | j \in J]$ whose elements we will write as linear combinations of monomials in the generators $y_{rs}^{\pm 1}$ for $r \neq i$ and  $\tau_{ij}$ for $j \neq i, i+1$. These generators all skew-commute, so $Z(\bar{A}_i)$ will be spanned by the monomials which commute with all the generators. First consider $y_{rs}$ for $r,s \neq i$. Then $y_{rs}$ commutes with all generators except $y_{sr}$ with which it skew-commutes by the primitive $\ell$-th root of unity $q_{ij}$. Thus any central monomial can be written without $y_{sr}$. Similarly, $y_{ji}$ for $j \neq i,i+1$ commutes with every generator except $\tau_{ij}$ so we see $Z(\bar{A}_i) \subseteq k[y_{ji}^{\pm 1}|j \in J]$. The remaining generators $\tau_{ij}$ act by conjugation on $k[y_{ji}^{\pm 1}|j \in J]$. If $H$ denotes the group they generate modulo scalars, then $Z(\bar{A}_i) = k[y_{ji}^{\pm 1}|j \in J]^H$. Intersecting this with $A_Q/(y_i)$ gives the invariant ring (\ref{eq:ramAQ}) as desired. 
\end{proof}

In the 3-dimensional prime index case, terminal toric ramification was classified \cite[Proposition~5.4]{11authors}. We re-write their result in our notation below. 

\begin{proposition} \label{prop:11authorTerminalToric}
Suppose now that $\ell$ is an odd prime and that $n=3$ so the multiplicatively skew-symmetric $Q=(q_{ij})$ is given by the triple $\vec{q}:=(q_{12},q_{23},q_{31})$. 
Then the Brauer class of $A_Q \otimes_R k[[x_1,x_2,x_3]]$ is terminal iff two of the coordinates of $\vec{q}$ are inverses of each other. 
\end{proposition}

We can now re-interpret this result using the ramification formula (\ref{eq:ramAQ}) of Proposition~\ref{prop:AQisMaximal} as follows. 
\begin{corollary} \label{cor:ramCoverToricTerminal}
Suppose $Q$ is a skew-symmetric $3 \times 3$-matrix with entries in $\mu_{\ell}$ where $\ell$ is an odd prime. Then $A_Q \otimes_R k[[x_1,x_2,x_3]]$ is terminal iff one of the ramification covers is Gorenstein, or equivalently, Kleinian of type $A_{\ell-1}$. 
\end{corollary}
\begin{proof}
We use Proposition~\ref{prop:AQisMaximal} to compute the ramification above the coordinate plane $x_1 = 0$. Now $H = \{(m_2,m_3)\}$ is a cyclic group of order $\ell$ and we pick the generator $\sigma = (-1,1)$. Its action on $k[y_{21},y_{31}]$ is given by 
\begin{equation*}
    \sigma \colon \ y_{21} \mapsto q_{21}^{-1} y_{21} = q_{12} y_{21}, \ \ y_{21} \mapsto q_{31} y_{31}. 
\end{equation*}
The ramification cover $k[y_{21},y_{31}]^{H}$ is Gorenstein precisely when this linear $H$-action factors through $SL_2$, which in this case, amounts to the fact that $q_{12},q_{31}$ are inverses of each other. 
\end{proof}

\section{Strange Duality}

In this section, we make a diversion to explore how the toric orders over the polynomial ring are strangely related to skew polynomial rings. 

Our basic setup will be similar to the previous section though we now insist that $\ell \in \bN$ is a prime. We fix a base field $k$ that contains primitive $\ell$-th roots of unity. Consider the lattice $M =  \bZ^n$.  Our skew-symmetric matrix $Q = (q_{ij})$ with $q_{ij} \in \mu_\ell$ will now also be viewed as  a skew symmetric pairing $Q : M \times M \to \mu_\ell$. 
We fix $\zeta$, a primitive $\ell$-th root of unity and write $P= \log_{\zeta} Q = (p_{ij})$ to be the skew symmetric matrix with entries in $\bZ/\ell$ such that 
$\zeta^{p_{ij}} = q_{ij}.$

Let $\sigma$ be the cone $\mathbb{N}^n$ and write $k[x_1,\ldots,x_n] = k[\sigma]$ with field of fractions $K$.
Consider first the Morita equivalent tensor products of symbol algebras determined by $Q$,
$$\bigotimes_{1 \leq i < j \leq n} (x_i,x_j)^K_{q_{ij}} \stackrel{\text{Morita}}{\sim} 
\bigotimes_{1 \leq i < j \leq n} ((x_i,x_j)^K_{\zeta})^{\otimes p_{ij}}$$
whose Brauer class we denote by  $\alpha_Q \in \mathrm{Br} K$ and let $A_Q$ denote a maximal order in this central simple algebra. If all the $q_{ij}$ are non-trivial, then this is given by (\ref{eq:maxToricOrder}). 

To $Q$ we can also associate the skew polynomial 
$$k_Q[x_1,\ldots,x_n] := k\langle x_1,\ldots,x_n \rangle/(x_jx_i - q_{ij}x_ix_j).$$
Let $$\rad Q = \{ m \in M \,|\, Q(m,n) =1 \mbox{ for all } n \in M\}$$
and $\rad_+ Q = \rad Q \cap \sigma$.
Recall that the centre of $k_Q[x_1,\ldots,x_n]$ is given by  
$$ Z(k_Q[x_1,\ldots,x_n] ) = k[\rad_+ Q].$$
Lastly, we define the dual lattices $N = M^\vee = \Hom_\bZ(M,\bZ)$ and $(\rad Q)^\vee$ and dual cones
$$\sigma^\vee = \{ n \in N \,|\, (n,m) \geq 0 \mbox{ for all } m \in \sigma \}$$ and similarly $(\rad_+ Q)^\vee$.

We note that the singularities of the dual of the centre of $k_Q[x_1,\ldots,x_n]$ and $A_Q$ share the following properties. 
\begin{theorem} Let $\ell \in \bN$ be prime.
    Let $Q : M \times M \to \mu_\ell$ be a skew symmetric pairing.  Then 
    the order $A_Q$ is terminal (respectively canonical, log terminal) if and only if the singularity $X_Q = \Spec k[(\rad_+ Q)^\vee]$ is terminal (respectively canonical, log terminal).
\end{theorem}    
\begin{proof}
To set up the proof we recall some facts and discuss how the discrepancy is computed.
Firstly, we note that 
$$\ell M \subseteq \rad Q \subset M$$
and we scale by $1/\ell$ to obtain
$$M \subseteq \frac{1}{\ell} \rad Q \subseteq \frac{1}{\ell}M$$
Contrary to usual conventions, this will be our lattice of one parameter subgroups.
Furthermore, for cones we have
$$\sigma \subset \frac{1}{\ell}\rad_+ Q \subset \frac{1}{\ell}\sigma$$
Write $e_i$ for the generators of $\sigma$ and note that the cone $\frac{1}{\ell}\rad_+Q$ is simplicial with
primitive generators for the rays $\bR_+ e_i$ given by $f_i = \frac{1}{\ell} e_i$ or $e_i$ depending on whether $e_i \in \rad Q$ or not. 
To compute the discrepancy of $\Spec k[(\rad_+ Q)^\vee]$ we consider all toric divisors coming 
from primitive elements $w \in \frac{1}{\ell}\rad_+ Q \subseteq \frac{1}{\ell}M$ except the divisors in the current model, namely the $f_i$.
Let $w=\sum w_i f_i.$  Then the discrepancy of $X_Q$ at the toric divisor determined by $w$ is given by 
$$a(X_Q,w)  = \sum w_i -1.$$
We first note that at least two $w_i \neq 0$.  Also, if some $w_i \geq 1$ then $a(X_Q,w) > 0$ and we can dispose of the divisor corresponding to $w$. In particular, if some $f_i =\frac{1}{\ell}e_i$ then $w_i$ is integral so we may assume that $w_i = 0$.  In other words, we are reduced to considering toric divisors corresponding to primitive rays in 
$$R = \left\{ w = \sum w_i e_i \in \rad Q \,|\ 0\leq w_i < 1 \mbox{ and } w_i = 0 \text{ if } f_i \neq e_i \right\} - \{ e_i\}.$$
The discrepancy formula thus reduces to 
$$ a(X_Q) = \min_{w \in R} 
 \sum w_i -1.$$

We now perform a similar calculation for the pair $(\bA^n,\alpha_Q)$.  To determine the discrepancy in this case we need the ramification index of $\alpha_Q$ at each toric divisor.
This can be computed using the tame ramification formula.  Let $K$ be a field with chosen primitive $\ell$-th root of unity $\zeta \in K$.  For $a,b \in K^{\times}$, let $A=(a,b)^K_\zeta$ be a symbol algebra as usual.  Fix a
divisorial valuation $\nu: K^{\times} \to \bZ$ and let $K_\nu$ be the residue field. The ramification of $A$ at $\nu$ is given by 
$$\rho_\nu(A) = (-1)^{\nu(a)\nu(b)} \overline{a^{\nu(b)} b^{-\nu(a)}} \in K_{\nu}^{\times} / K_{\nu}^{\times \ell}$$
where $\overline{a}$ is the reduction to $K_{\nu}^{\times}/ K_{\nu}^{\times \ell}$.
The ramification index $r_{\nu}(A)$ of the symbol $A$ at $\nu$ is the order of $\rho_{\nu}(A)$ in $K_{\nu}^{\times}/ K_{\nu}^{\times \ell}$. Given our assumption that $\ell$ is prime, this ramification index is either 1 or $\ell$. 

We consider toric divisors over $\bA^n$ which correspond to primitive vectors $w = \sum w_i e_i \in \sigma$.
Note that we can interpret $w$ as a divisorial valuation on $K$ and so let $\rho_w, r_w$ denote the corresponding ramification and ramification index of $A_Q$. Furthermore, the divisor corresponding to $w$ is toric containing a dense torus which can be naturally identified with $\Spec k[w^{\perp}]$. Thus the corresponding residue field $K_w$ is the fraction field of $k[w^{\perp}]$. The tame ramification formula now gives
$$\rho_w(\alpha_Q) = \prod_{1 \leq i < j \leq n} ((-1)^{w_iw_j} \overline{x_i^{w_j} x_j^{-w_i}} )^ {p_{ij}}$$
and so $r_w=1$ if and only if $w \in \rad Q$. In particular, $\rho_i := \rho_{e_i} = 1$ if and only if $f_i \neq e_i$. 

The b-discrepancy of the pair $(\bA^n,\alpha_Q)$ at the divisor associated to $w$ is given by the formula
$$b(\bA^n,\alpha_Q, w) = \sum w_i \frac{r_w}{r_i} -1.$$
Note as above that if $r_i =1$ and $w_i \neq 0$ then $b(\bA^n,\alpha_Q, w)>0$, so we can assume that $w_i=0$ whenever $f_i \neq e_i$. Also, if $r_w=\ell$ we also have $b(\bA^n,\alpha_Q, w)>0$.  
So to compute the discrepancy we obtain
$$b(\bA^n,\alpha_Q) = \min_{ \frac{w}{\ell} \in R} \sum  \frac{w_i}{\ell} -1.$$
\end{proof}

\section{Deformations of symbols}  \label{sec:deformSymbols}

We turn now to constructing  terminal orders ramified on surfaces with Kleinian singularities. Our approach is to consider deformations of symbol algebras which can be thought of as higher rank generalisations of Clifford algebras. 

Let $R$ be a commutative domain containing a primitive $\ell$-th root of unity $\zeta$ and $a,b\in R$ be non-zero elements and $r \in R$. We define the {\em deformed symbol $R$-algebra} 
\begin{equation}  \label{eq:defSymbol}
    (a,b)^R_{\zeta,r} = \frac{R \langle x,y \rangle}{(x^\ell -a, y^\ell-b, yx -\zeta xy - r)}
\end{equation}
and refer to $x,y$ as {\em standard generators}. Note that when $\ell=2$, this is just a Clifford $R$-algebra and the usual hypothesis that the characteristic is not 2 has been translated to the condition that $R$ has a primitive square root of unity. 

\begin{proposition}  \label{prop:defSymbolBasis}
The $R$-algebra $A = (a,b)^R_{\zeta,r}$ above is free as an $R$-module with basis $\{x^i y^j\,|\, 0 \leq i,j<n\}$. 
\end{proposition}
\begin{proof}
Our relations mean that we can write elements of $A$ as left $R$-linear combinations of $x^i y^j$ for  $0 \leq i,j<\ell$. We use Bergman's diamond lemma to show they are a free basis. We need to check overlaps for $y^\ell x$ and $y x^\ell$. For example,
\begin{multline*}
bx = y^\ell x = y^{\ell-1} (\zeta x y + r) 
= \zeta y^{\ell-2}(\zeta x y +r)y + y^{\ell-1}r \\ 
= \zeta^2 y^{\ell-2} x y^2 + (\zeta + 1) y^{\ell-1} r 
= \ldots  \\ = xy^\ell + (\zeta^{\ell-1} + \ldots + \zeta + 1)y^{\ell-1}r 
= xy^\ell = xb
\end{multline*}
so no relations between the $x^iy^j$ arise from this overlap. A similar calculation disposes of the overlap $y x^e$. 
\end{proof}

Fix $A = (a,b)^R_{\zeta,r}$ and let $p = xy + \frac{1}{\zeta - 1}r \in A$. Then one calculates the following skew-commutation relations:
\begin{equation} \label{eq:pSkewCommutes}
px = \zeta x^2 y + \frac{\zeta}{\zeta-1} rx = \zeta xp, \quad yp = \zeta x y^2 + \frac{\zeta}{\zeta-1} ry = \zeta p y.
\end{equation}

\begin{proposition} \label{prop:defSymbolGenericallySymbol}
 Let $A = (a,b)^R_{\zeta,r}$ where $R$ is a commutative domain, $\zeta \in R$ is a primitive $\ell$-th root of unity  and $ab \notin (r)\triangleleft R$. If $K$ is the field of fractions of $R$ then $p^\ell \in K$ and 
 $$K\otimes_R A \simeq (a,p^\ell)^K_{\zeta} \simeq (p^\ell,b)_{\zeta}^K$$
 and is a central simple $K$-algebra.
\end{proposition}
\begin{proof}
Consider the subfield $K[x] \subset K \otimes_R A$ and note that $V = K[x] + K[x]y = K[x] + K[x] p$. Taking powers of $V$, we see that 
$$K\otimes_R A = K[x] \oplus K[x] p \oplus \ldots \oplus K[x] p^{\ell-1}.$$
From the skew-commutation relations Equation~(\ref{eq:pSkewCommutes}), this is the eigenspace decomposition with respect to conjugation by $x$. Examining the conjugation by $p$ action, we see that $Z(K \otimes_R A) = K$. Note that $p^\ell$ commutes with both $x$ and $p$ so is central and we thus have $K\otimes_R A \simeq (a,p^\ell)^K_{\zeta}$. An analogous proof shows that $K\otimes_R A \simeq (p^\ell,b)_{\zeta}^K$. 

It remains only to prove $p^\ell$ is non-zero as follows. We first expand $p^\ell = (xy + \frac{1}{\zeta -1}r)^\ell$ using the binomial theorem and then use the relation $yx = \zeta xy + r$ to rewrite $p^\ell$ as an $R$-linear combination of the basis elements $x^i y^j$ as in Proposition~\ref{prop:defSymbolBasis}. The ``middle terms'' ${\ell \choose i} (xy)^i (\zeta-1)^{i-\ell}r^{\ell-i}$, where  $0<i<\ell$ will only produce terms of the form $x^jy^j$ for $0<j<\ell$ so by centrality of $p^\ell$ must cancel with other terms in $p^\ell$. Similarly, all terms of $(xy)^\ell$ cancel except  
$$\zeta^{\ell \choose 2} x^\ell y^\ell = (-1)^{\ell-1}ab$$
so $p^\ell = (-1)^{\ell-1}ab + (\zeta-1)^{-\ell}r^{\ell}$. Our assumption that $ab \notin (r)$ ensures then that $p^\ell \neq 0$. 
\end{proof}
In the course of the proof we established the following result.
\begin{lemma} \label{lem:ptoe}
In $(a,b)^R_{\zeta,r}$ where $\zeta$ is a primitive $\ell$-th root of unity, we have 
$$p^\ell = (-1)^{\ell-1}ab + (\zeta-1)^{-\ell}r^{\ell}.$$
\end{lemma}

\begin{theorem}  \label{thm:defSymbolRamification}
Let $R$ be a commutative normal domain, $\zeta \in R$ a primitive $\ell$-th root of unity and $A = (a,b)^R_{\zeta,r}$. Suppose the following hold:
\begin{itemize}
    \item $ab\notin (r)$,
    \item the codimension of $(a,b) \triangleleft R $ is $> 1$ and,
    \item $c:=(-1)^{\ell-1}ab + (\zeta-1)^{-\ell}r^{\ell}$ is prime. 
\end{itemize}
Then 
\begin{enumerate}
    \item $A$ is Azumaya away from the zeros of the ideals $(a,b)$ and $(c)$.
    \item The element $p = xy + \frac{1}{\zeta -1}r \in A$ is normal and 
    $$\Abar:= A/(p) \simeq \frac{R[x,y]}{(x^\ell-a,y^\ell-b,(\zeta -1)xy + r)}.$$ 
    \item If $\Abar$ is a domain, then $A$ is a maximal order whose ramification above $(c)$ is given by $\Abar$. 
    \item If $\Abar$ is a regular domain, then $A$ has finite global dimension. 
\end{enumerate}
\end{theorem}
\begin{proof}
We know from Propositions~\ref{prop:defSymbolBasis} and \ref{prop:defSymbolGenericallySymbol} that $A$ is a reflexive order in a central simple $K$-algebra. As in the proof of Proposition~\ref{prop:defSymbolGenericallySymbol}, the skew-commutation relations~(\ref{eq:pSkewCommutes}) and Lemma~\ref{lem:ptoe} show that 
\begin{equation} \label{eq:Abinverse}
A[b^{-1}] = R[b^{-1}]\langle x,p\rangle \simeq (a,c)^{R[b^{-1}]}_{\zeta}
\end{equation}
which is Azumaya away from the zeros of $a$ and $c$. Considering $A[a^{-1}]$ in the same way proves (1).

To prove (2), note first that the skew-commutation relations~(\ref{eq:pSkewCommutes}) ensure that $p$ is normal. Furthermore, the relation $yx = \zeta xy + r$ in $A$ can be re-written as $yx - xy = (\zeta-1)xy + r = (\zeta-1) p$ so (2) follows as does (4). 

We now turn to proving that $A$ is maximal under the assumption in (3). 
Since we know $A$ is reflexive and, by part~(1) Azumaya at every codimension prime except $(c)$, it suffices to show that $A \otimes_R R_{(c)}$ is maximal. First note that, as in Equation~(\ref{eq:Abinverse})
$$A \otimes_R R_{(c)} \simeq (a,c)^{R_{(c)}}_{\zeta}$$
is hereditary since it contains the regular normal element $p$ and the quotient $(A \otimes_R R_{(c)})/(p)$ is a field. By \cite[Theorem~2.3]{AG} and our hypothesis on $\Abar$, we see that $A \otimes_R R_{(c)}$ and hence $A$ is maximal. Furthermore, the ramification above $(c)$ is given by the quotient $\Abar$. 
\end{proof}

\section{Some terminal orders ramified on with surfaces Kleinian singularities}  \label{sec:terminalRamOnKleinSing}

In this section, we use the deformed symbol algebras studied in Section~\ref{sec:deformSymbols} to construct some terminal orders over $R = k[[u,v,w]]$ with non-toric ramification. 

The search for terminal ramification is guided by the following key observation. 

\begin{proposition}  \label{prop:ramOnADEisTerminal}
Let $A$ be a maximal $R$-order which is ramified only on a surface with a Kleinian singularity $D \subset \Spec R$. Then $A$ is terminal. 
\end{proposition}
\begin{proof}
First note that the log pair $(\Spec R, D)$ is canonical by \cite[Theorem~5.34]{KM} or more, accurately, its proof. (Indeed, the proof they give that $D$ Kleinian $\implies (\Spec R,D)$ canonical does not use the hypothesis that $D$ is a general hypersurface section). We conclude that the log pair $(\Spec R, \Delta_A)$ associated to $A$ is terminal by \cite[Corollary~2.35(4)]{KM} and the fact that $\Delta_A = (1 - \frac{1}{e})D$ where $e$ is the ramification index. Finally, it follows that $A$ is also terminal by \cite[Remark~2.19]{11authors}.
\end{proof}

Suppose now that $A$ is a maximal $R$-order ramified only along a Kleinian singularity $D \subset \Spec R$. The ramification data is given by a cyclic cover $\Dtilde \to D$ which must be unramified in codimension one by the coniveau spectral sequence. The possibilities of such are bounded by the \'etale fundamental group of $D - 0$ where $0$ is the closed point. Now $D = \Spec k[[x,y]]^G$, for some finite subgroup $G < SL_2$ so we must have $\Dtilde = \Spec k[[x,y]]^N$ where $N \triangleleft G$ is a normal subgroup with cyclic quotient $G/N$. These possibilities are quite constrained (see for example from \cite[Section~3]{CI21}). We list them below using the standard ADE notation. The third column indicates the ramification index $\ell$ and the parameter $n$ is a positive integer. In most cases, the group $G$ has cyclic abelianisation $G^{ab} = G/[G,G]$ so the subgroup $N$ is determined by its index. The only exception is the case of the binary dihedral group $G$ of order $4m$ where $m$ is even. In this case $G^{ab}$ is the Klein 4 group. Note there are two rows with $G = D_{2n}, N\simeq D_{n+1}$ corresponding to the two choices for the actual subgroup $N$. The columns headed with $a,b,r$ will be referred to in the Corollary~\ref{cor:constructTypeA} below. 

\begin{equation}  \label{eq:tableOfGroups}
\begin{array}{c|c|c|c|c|c}
    G & N & \ell & a & b & r \\
    \hline 
    A_{n\ell-1} & A_{n-1} & \ell & u & v & w^n \\
    D_{n+2} & A_{2n-1} & 2 & v^2+w^n & w & u \\
    D_{2n} & D_{n+1} & 2 & (v+w)^{n-1}  & (v-w^{n-1})w & u \\
    D_{2n} & D_{n+1} & 2 & (v+w)^{n-1}w & (v-w^{n-1}) & u \\
    D_{2n+3} & A_{2n} & 4 & ? & ? & ? \\
    E_6 & D_4 & 3 & u^2+v & u^2-v & w  \\
    E_7 & E_6 & 2 & v & v^2-u^3 & w 
\end{array}
\end{equation}
We seek to construct maximal orders with these ramification data via the deformed symbol algebra construction. 

\begin{corollary}  \label{cor:constructTypeA}
Let $R = k[[u,v,w]]$, and let $\zeta\in k$ be a primitive $\ell$-th root of unity and $A$ be the deformed symbol algebra $(a,b)^R_{\zeta,r}$ defined in Equation~(\ref{eq:defSymbol}) with standard generators $x,y$. 
\begin{enumerate}
\item \label{item:typeA} If $n$ is a positive integer, $a=u, b=v, r = w^n$ then $A$ is a terminal order ramified on a Kleinian singularity of type $A_{n\ell-1}$ and whose ramification cover is the $\ell$-fold cover, unramified in codimension one given by a type $A_{n-1}$ singularity.  

Furthermore, if $n=1$, then $A$ has global dimension three. 
\item \label{item:typeDA} Let $\ell=2, a=v^2+w^n, b=w, r=u$. Then $A$ is a terminal order ramified on a Kleinian singularity of type $D_{n+2}$ with ramification cover of type $A_{2n-1}$. 
\item \label{item:typeDD} We consider two cases depending on whether the index $i$ below is 0 or 1. Let $\ell=2, a=(v+w^{n-1})w^i, b = (v-w^{n-1})w^{1-i}, r=u$. Then $A$ is a terminal order ramified on a Kleinian singularity of type $D_{2n}$ whose ramification cover is a type $D_{n+1}$ singularity.
\item \label{item:typeE6} If $\ell=3, a = u^2+v, b=u^2-v, r= w$, then $A$ is a terminal order ramified on a Kleinian singularity of type $E_6$. 
\item \label{item:typeE7} If $\ell=2, a=v, b=v^2-u^3, r=w$ then $A$ is a terminal order ramified on a Kleinian singularity of type $E_7$.
\end{enumerate}
\end{corollary}
\begin{proof}
In each case, the theorem will follow from Theorem~\ref{thm:defSymbolRamification} and Proposition~\ref{prop:ramOnADEisTerminal} after computing 
$$ \Abar := A/(xy + \frac{1}{\zeta -1}r)$$
and verifying that it is a domain giving the correct ramification data. 

In case~(\ref{item:typeA}) we find 
$$\Abar := A/(xy + \frac{1}{\zeta -1}w^{n}) \simeq k[[x,y,w]]/((\zeta -1)xy + w^n)$$
which is the Kleinian singularity of type $A_{n-1}$ so in particular, regular when $n=1$. Hence $A$ is a maximal order which is ramified on the locus $0=(-1)^{\ell-1}uv + (\zeta -1)^{-\ell}w^{n\ell}$ which is a type $A_{n\ell-1}$ Kleinian singularity. 

In case~(\ref{item:typeDA}), we have 
$$\Abar \simeq k[[x,y,v]]/(x^2-v^2-y^{2n})$$
which defines a type $A_{2n-1}$ singularity and the discriminant here is $\frac{1}{4}u^2 - (v^2+w^n)w$ which defines a type $D_{n+2}$ singularity. 

In Case~(\ref{item:typeDD}), the two possibilities for $i$ are similar and we only consider here the case where $i=0$. Then 
$$\Abar \simeq k[[x,y,w]]/(y^2-(x^2-2w^{n-1})w)$$
which is a type $D_{n+1}$ singularity. The discriminant here is $\frac{1}{4}u^2 - (v+w^{n-1})(v-w^{n-1})w $ which defines a type $D_{2n}$ singularity. 

In case~(\ref{item:typeE6}) we have 
$$ \Abar \simeq k[[x,y,u]]/(y^3+x^3 - 2u^2)$$
which is the type $D_4$ Kleinian singularity. The discriminant here is 
$$(u^2+v)(u^2-v) + (\zeta -1)^{-3}w^3 = u^4-v^2 + (\zeta -1)^{-3}w^3$$
which defines a type $E_6$ Kleinian singularity.

In case~(\ref{item:typeE7}), we have 
$$\Abar \simeq k[[x,y,u]]/(y^2-x^4+u^3)$$
which is a type $E_6$ singularity and the discriminant is $\frac{1}{4}w^2-v(v^2-u^3)$ which defines a type $E_7$ singularity. 
\end{proof}

Note that the Corollary covers all the rows of the table in Equation~(\ref{eq:tableOfGroups}) except where $G$ is type $D_{2n+3}$ and $N$ is type $A_{2n}$. It seems this case cannot be constructed via deformed symbol algebras like the others.

\section{The even Clifford Sklyanin algebra}  \label{Sklyanin}

Our final example of a terminal 3-fold order comes from Sklyanin algebras. Let $A$ be an AS-regular algebra of dimension three constructed using a 2-torsion point of an elliptic curve. These can be conveniently constructed via Clifford algebras as follows.

We start with the polynomial ring $R = k[u,v,w]$ which we consider a graded algebra where the generators $u,v,w$ have degree two. Let $Q \colon V \times V \to R_1 = ku \oplus kv \oplus kw$ be a quadratic form on the three-dimensional vector space $V = kx \oplus ky \oplus kz$ with values linear forms in $u,v,w$. We suppose that the cubic form $\det Q$ defines a smooth elliptic curve in $\mathbb{P}^2_{u,v,w}$ and that $Q$ is {\em non-degenerate} in the sense that its $R$-bilinear extension induces an injective morphism $R \otimes_k V \to (R \otimes_k V)^*$.
Consider the free $R$-algebra $T(V) = R\langle x,y,z \rangle$ which we grade by setting the degrees of the generators $x,y,z$ to be one. As usual, we define the {\em Clifford algebra}
\begin{equation}
Cl(Q) := T(V) / \mathcal{I}
\end{equation}
where $\mathcal{I}$ is the two-sided ideal generated by the relations $v_1v_2 + v_2v_1 - 2Q(v_1,v_2)$ for $v_1,v_2 \in V$. Since these relations all have even degree, $Cl(Q)$ is $\bZ/2\bZ$-graded and we can consider its even graded part $Cl_0(Q)$, the {\em even Clifford algebra}. The following proof will use some familiar facts about Clifford algebras that can for example be found in \cite{CI12}.

\begin{proposition} \label{prop:Sklyanin}
The even Clifford algebra $A = Cl_0(Q)$ above is a terminal $R$-order.
\end{proposition}
\begin{proof}
First note that $A$ is a reflexive order on $\Spec R = \bA^3$ which, by assumption on $\det Q$ is ramified on a cone $D$ over an elliptic curve in $\bP^2$. Furthermore, $A$ is hereditary in codimension one. The ramification is given by a double cover of $D$, \'etale away from the cone point. The log variety of $A$ is thus $(\bA^3, \frac{1}{2}D)$ and it suffices by \cite[Remark~2.19]{11authors} to show that this is terminal. Let $f \colon Z \to \bA^3$ be the blowup at the cone point $0$, $E$ be the exceptional $\bP^2$ and $D'$ be the strict transform of $D$. Now
\begin{equation}  \label{eq:SklyaninDiscrep}
f^*(K_{\bA^3} + \frac{1}{2}D) - K_Z \equiv \frac{1}{2}D' - \frac{1}{2}E
\end{equation}
so by \cite[Proposition~2.36]{KM}, any exceptional divisor (in some resolution of $\bA^3$) with non-positive discrepancy must have centre on $Z$. However, Equation~(\ref{eq:SklyaninDiscrep}) shows the discrepancy of $E$ is $\frac{1}{2}$.

\end{proof}

\section{Uniqueness results in the regular case}

We show that if $A$ is a local Cohen-Macaulay maximal order in $Q$ of global dimension three over commutative noetherian normal local domain $R$, then any other maximal order $B$ in $Q$ which is also Cohen-Macaulay is Morita equivalent to $A$. We thank Michael Wemyss who suggested this result which holds in the commutative case. We follow \cite[Section~4]{AG} closely only adding details where there are some extra noncommutative considerations. 

Throughout this section, $R$ is a commutative noetherian normal local domain and $A$ is an $R$-algebra which is finite as an $R$-module. We will often use the concepts of {\em Cohen-Macaulay} (by default always maximal for us), {\em reflexive} and {\em depth} for $A$-modules, by which we just mean, the corresponding concepts for their underlying $R$-module. Recall also that the local cohomology functor $H^i_{\frakm}(?)$ can be used to detect depth. 

We start with a presumably well-known noncommutative analogue of the Auslander-Buchsbaum formula.
\begin{proposition}  \label{prop:AuslanderBuchsbaum}
Suppose $A$ is a (maximal) Cohen-Macaulay $R$-algebra with finite global dimension $\text{gl.dim} A = d := \dim R$. For any finitely generated $A$-module $M$ with finite projective dimension we have
$$\text{pd}\, M + \text{depth}\, M = d.$$
\end{proposition}
\begin{proof}
We argue by induction on $i:= \text{pd}\, M$, the case $i=0$ holding since $A$ is Cohen-Macaulay. Suppose now that $i>0$ and consider a short exact sequence of finite $A$-modules
$$0 \to K \to P \to M \to 0$$
where $P$ is projective. When $i=1$, the long exact sequence in local cohomology shows that $H^j_{\frakm}(M) = 0$ when $j < d-1$ so $\text{depth}\, M \geq d-1$. Also, \cite[Proposition~3.5]{Ramras} states that every maximal Cohen-Macaulay $A$-module must be projective so we must have $\text{depth}\, M = d-1$. Induction using dimension shifting and the long exact sequence in local cohomology as above now completes the proof. 
\end{proof}

We also need a version of \cite[Proposition~4.10]{AG} which detects projective dimension. 

\begin{proposition}  \label{prop:detectProjDim}
Let $A$ be a finite $R$-algebra and $E,M$ finitely generated $A$-modules such that $\text{pd}_A E = n$ and that every simple $A$-module occurs as a quotient of $M$. Then $\Ext_A^n(E,M) \neq 0$. 
\end{proposition}
\begin{proof}
Since completion is an exact functor, we may assume that $R$ is complete and hence, $A$ is semiperfect. Thus the projective covers of simple modules have the form $e_iA$ for some complete set of primitive idempotents $e_1, \ldots, e_s \in A$. Furthermore, we may consider a minimal projective resolution 
$$0 \to P_n \to P_{n-1} \to \ldots \to P_0 \to E \to 0.$$
Expressing the $P_j$'s as direct sums of $e_i A$'s, the differentials are given by matrices with entries in $e_{i_1} J e_{i_2}$ where $J = \text{rad}\, A$ and $i_1,i_2$ are appropriate indices which can vary. Let $\Phi$ be the matrix representing the map $P_n \to P_{n-1}$.  Then, expressing 
$\Hom_A(P_n,M)$ and $\Hom_A(P_{n-1},M)$ as direct sums of $Me_i$, the map $\Hom_A(P_{n-1},M) \to \Hom_A(P_n,M)$ is given by right multiplication by the matrix $\Phi$. Our assumption on $M$ ensures that $\Hom_A(P_n,M)$ has a non-zero summand $Me_j$ say. To show that $\Phi$ is not surjective and so prove the proposition, it suffices to show given any column vector $\Phi$ in $(Je_j)^{\oplus N}$, the induced map $M^{\oplus N} \to Me_j$ is not surjective. However, the cokernel of this map has $Me_j/MJe_j \simeq (M/MJ)e_j$ as a quotient which is non-trivial by our assumption on $M$.
\end{proof}

Finally we need a version of \cite[Proposition~4.9]{AG}. We omit the proof since Auslander-Goldman's proof extends without change to our setting. 
\begin{proposition}  \label{prop:Ext1EE}
Let $A$ be a Cohen-Macaulay $R$-algebra with finite global dimension equal to $\dim R$. If $E$ is a reflexive $A$-module with $\Hom_A(E,E)$ Cohen-Macaulay, then $\Ext^1_A(E,E) = 0$.
\end{proposition}

\begin{proposition}  \label{prop:NCCR}
Let $R$ be a commutative normal noetherian local domain of dimension 3 and $A$ be a Cohen-Macaulay $R$-algebra with finite global dimension $3$. Let $E$ be a finite $A$-module such that 
\begin{enumerate}
    \item $\Hom_A(E,E)$ is Cohen-Macaulay
    \item every simple $A$-module appears as a quotient of $E$. 
\end{enumerate}
Then $E$ is projective iff it is reflexive. 
\end{proposition}
\begin{proof}
Only the reverse implication is not clear so we assume $E$ is reflexive and hence $\text{depth}\, E \geq 2$. By the Auslander-Buchsbaum formula in Proposition~\ref{prop:AuslanderBuchsbaum}, we have $\text{pd}\, E \leq 1$. Suppose by way of contradiction that $\text{pd}\, E = 1$. Then Proposition~\ref{prop:detectProjDim} shows that $\Ext^1_A(E,E) \neq 0$ contradicting Proposition~\ref{prop:Ext1EE}. 
\end{proof}

\begin{remark}
In \cite[Theorem~4.4]{AG} which inspired this result, there is no dimension restriction. We cannot recover Auslander-Goldman's result in this setting as they use induction via generic hyperplane cuts. However, these may not preserve finite global dimension in the noncommutative case. 
\end{remark}

\begin{corollary}   \label{cor:unique}
Let $A$ be a Cohen-Macaulay order in a central simple algebra $Q$ with global dimension 3 over a commutative noetherian normal local domain $R$ of dimension 3. If $A$ is local in the sense that $A/\text{rad}\, A$ is simple, then any Cohen-Macaulay maximal order $B$ in $Q$ is Morita equivalent to $A$. 
\end{corollary}
\begin{proof}
We can consider the $(B,A)$-bimodule $BA \subset Q$. Let $E$ be its $R$-reflexive hull in $Q$ which is still a $(B,A)$-bimodule so there is a natural injective ring homomorphism $B \hookrightarrow \End_A E$. Since $\End_A E$ is also an order in $Q$, maximality ensures that $B \simeq \End_A E$. Thus the hypotheses of Proposition~\ref{prop:NCCR} hold and $E$ is a projective generator. 
\end{proof}
\begin{remark}  \label{rem:typeAunique}
The corollary applies to the terminal orders described in the $n=1$ case of Corollary~\ref{cor:constructTypeA}. 
\end{remark}

\bibliographystyle{amsalpha}

\bibliography{references}

\end{document}